\documentclass[reqno]{amsart}
\usepackage{amssymb}
\usepackage{hyperref}
\usepackage{setspace}
\usepackage{enumerate}

 \newtheorem{theorem}{Theorem}[section]
 \newtheorem{corollary}[theorem]{Corollary}
 
 \newtheorem{proposition}[theorem]{Proposition}
 
 \newtheorem{remark}[theorem]{Remark}
 \newtheorem{remarks}[theorem]{Remarks}
 \newtheorem{example}{Example}
 \numberwithin{equation}{section}
\newcommand{\e}{{\mathrm{e}}}
\newcommand{\im}{{\mathrm{Im}}}
\begin{document}

\title[The explicit matrix exponential]
 {Some Explicit Formulas for Matrix Exponential, Matrix Logarithm, the $n$th Power of Matrices and their Drazin Inverses}

\author{ Mohammed Mou\c{c}ouf and Said Zriaa}
\address{ Mohammed Mou\c{c}ouf and Said Zriaa,
University Chouaib Doukkali.
Department of Mathematics, Faculty of science
Eljadida, Morocco}
\email{moucou@hotmail.com}
\email{saidzriaa1992@gmail.com}


\date{}

\begin{abstract}
In this work, new closed-form formulas for the matrix exponential are provided. Our method is direct and elementary, it gives tractable and manageable formulas not current in the extensive literature on this essential subject. Moreover, others are recuperated and generalized. As a consequence, we easily obtain the Chevalley--Jordan decomposition and the spectral projections of any matrix. In addition, closed-form expressions for the arbitrary positive powers of matrices and their Drazin inverses are presented. Using these results, an elegant explicit formula for logarithm of matrices is obtained. Several particular cases and examples are formulated to illustrate the methods presented in this paper.
\end{abstract}

\maketitle
\section{Introduction}
The theory of functions of matrices plays a central role in linear algebra and it is widely used in various applications in pure and applied mathematics. This theory was subsequently developed and has been extensively studied by many mathematicians. Among these functions, the most interesting are those corresponding to the scalar functions $\e^{x}$, $\log(x)$, $\frac{1}{x}$ and $x^{k}$ where $k$ is an integer.
In this article, we are interested in the computation of above listed functions of matrices.\\
\indent Section~\ref{sec:2} is devoted to the matrix exponential. The computation of this matrix has been studied by many researchers and has attracted great attention in diverse papers in the literature. For such calculation, several methods and algorithms have appeared in the last few decades (see for example~\cite{Tmapo,Bellman,Harri,Eileo,Liz,Moler,Howland}). In~\cite{Moler}, authors presented a careful investigation on various such efforts they attempted to describe all the methods that seem to be practical. J. L. Howland~\cite{Howland} presented a procedure that generalizes a method described in~\cite{Moler}. T. M. Apostol~\cite{Tmapo} presented an elegant and manageable approach that gave explicit formulas for some special cases. His method does not produce the general case when the characteristic polynomial of the matrix has multiple roots.\\
\indent Much of the difficulty of the computation of the matrix exponential is bypassed by the algorithm of Putzer~\cite{Putzer}. In~\cite{Eileo}, I. E. Leonard presented an alternative method intending to minimize the mathematical prerequisites. His approach requires exactly the solution of homogeneous linear differential equations with constant coefficients. To avoid solving the initial value problems for these differential equations, E. Liz~\cite{Liz} provided a method which requires only the knowledge of a basis for each solution space of these equations.\\
\indent Because of the ubiquitous appearance of the matrix exponential in linear dynamical systems and other important applications, our purpose is to present explicit formulas to simplify its presentation and practice. In addition, these formulas should be beneficial in future investigations involving the matrix exponential.\\
\indent Section~\ref{sec:3} is devoted to the powers and the Drazin inverse of matrices. The powers of matrices are of considerable importance and have a wide range of applications. There is a deep relation between the powers of matrices and the Drazin inverse (see, Mou\c{c}ouf~\cite{Mmou}). The Drazin inverse is of intrinsic interest, it has important properties that make it extremely useful and interesting in various applications including cryptography, Markov chains, functional analysis, differential-algebraic equations, stochastic process, multibody system, Control theory (see for example~\cite{Campbell,Buc,Weiy} and references therein). Many authors have investigated the Drazin inverse from different viewpoints. Here, we propose an interesting explicit expression for the Drazin inverse for matrices and their powers in terms of the eigenvalues. The advantage of our method is that it gives explicit tractable representations for the Drazin inverse of a matrix and needs only calculating certain polynomials of this matrix.\\
\indent In section~\ref{sec:4}  we consider the logarithm of matrices. This function of matrices appears in many parts of mathematics, applied natural sciences, and engineering. It has many applications in the study of Systems Theory and has received much interest in Control Theory (see~\cite{Njhigham,Cdahlbrandt} and references therein). If the matrix $A$ has no eigenvalues on the closed negative real axis $\mathbb{R}^{-}$, then $A$ has a unique logarithm with eigenvalues in the set $\lbrace z\in \mathbb{C}/-\pi<\im(z)<\pi\rbrace$. This unique logarithm is called the principal logarithm of $A$. This logarithm is of great interest and is needed in many applications.\\
\indent While there are various methods for the computation of the matrix exponential, relatively few ones exist for the matrix logarithm. The exact computation of the matrix logarithm reveals considerable difficulties. Most of the methods presented for computing such matrix are approximation methods~\cite{Njhigham,Cdahlbrandt}. Our motivation comes from the fact that few works have been considered for exhibiting compact formulas for the matrix logarithm. It is important to emphasize that other methods for computing the matrix logarithm require advanced theory such as matrix square roots, Schur decompositions, and Pad\'{e} approximants (see~\cite{Jamarrero,Jrcardoso2,Njhigham,Cdahlbrandt} and references therein).\\
\indent In a previous work, the first-named author of this paper has presented a method for obtaining easily from any given $\mathcal{P}$-canonical form of an arbitrary nonsingular $k\times k$ matrix $A$ with eigenvalues $\alpha_{1},\ldots,\alpha_{k}$, both a logarithm $B$ of $A$, the eigenvalues of $B$ may be chosen in advance arbitrarily as $\log(\alpha_{1}),\ldots,\log(\alpha_{k})$, and one of its $\mathcal{P}$-canonical forms (see, Mou\c{c}ouf~\cite{Mouc22}). This result is the key for obtaining our logarithm matrices formulas.\\
We propose simple, direct, and compact formulas to compute the logarithms and the principal logarithm of matrices without any restrictions on the norm. We note that the Putzer matrix representation of the logarithm of complex matrices~\cite{Cdahlbrandt}, or real matrices~\cite{Jrcardoso2}, requires the exact computations of some rational integrals. We also note that in a recent work~\cite{Jamarrero}, the authors are interested in developing the exact computation for the principal logarithm of matrices. More precisely, under some conditions on the norm, they compute exactly the principal logarithm matrix, but in this method it is necessary to solve a system of linear recursive equations, determine the Fibonacci-H\"{o}rner decomposition of the matrix, investigate the properties of generalized Fibonacci sequences and their Binet formula. In addition, their approach requires laborious calculations of some functions. Contrary to these works and other many works on this subject, our method does not require any cumbersome and laborious calculations. Additionally, our approach is general. The attractive feature of our proposed method lies in the possibility of choosing in advance the eigenvalues of logarithms of a matrix, and hence easily finding the principal logarithm of matrices.\\
\indent Among the essential tools used in this paper are the Vandermonde matrices and their inverses and certain polynomials. So it is convenient to fix some notations.\\
\indent Let $\alpha_{1},\alpha_{2},\ldots,\alpha_{s}$ be distinct elements of $\mathbb{C}$ and $m_{1},m_{2},\ldots,m_{s}$ be nonnegative integers. For a non-constant polynomial $P(x)=(x-\alpha_{1})^{m_{1}}(x-\alpha_{2})^{m_{2}}\cdots(x-\alpha_{s})^{m_{s}}$ of degree $n$, we denote by $L_{jk_{j}}(x)[P]$ the following polynomial
\begin{equation}\label{eq: 1.1}
L_{jk_{j}}(x)[P]=P_{j}(x)(x-\alpha_{j})^{k_{j}}\sum_{i=0}^{m_{j}-1-k_{j}}\frac{1}{i!}g^{(i)}_{j}(\alpha_{j})(x-\alpha_{j})^{i}
\end{equation}
where $1\leq j \leq s,\,\,0\leq k_{j} \leq m_{j}-1$ and
\begin{equation*}
P_{j}(x)=\prod_{i=1,i\neq j}^{s}(x-\alpha_{i})^{m_{i}}=\frac{P(x)}{(x-\alpha_{j})^{m_{j}}},1\leq j \leq s
\end{equation*}
and
$$ g_{j}(x)=(P_{j}(x))^{-1}$$
Here and further $L^{(l)}_{jk_{j}}(x)[P]$ means the $l$th derivative of $L_{jk_{j}}(x)[P]$.\\
According to~\cite{As1}, we can write every polynomial $Q$ of degree less than or equal to $n-1$ as
\begin{equation}\label{eq: 1.3}
Q=\sum_{j=1}^{s}(\sum_{k_{j}=0}^{m_{j}-1}\frac{1}{k_{j}!}Q^{(k_{j})}(\alpha_{j})L_{jk_{j}}(x)[P])
\end{equation}
This formula is of great importance, it is used in~\cite{Mmous} to invert the confluent Vandermonde matrix.
In~\cite{Mmous}, the confluent Vandermonde matrix associated with the polynomial $P$ of degree $n=m_{1}+m_{2}+\cdots+m_{s}$
\begin{equation*}
P(x)=(x-\alpha_{1})^{m_{1}}(x-\alpha_{2})^{m_{2}}\cdots(x-\alpha_{s})^{m_{s}},
\end{equation*}
is defined to be the following block matrix
\begin{equation}\label{eq: 1.4}
V_{G}(P) = (V_{1}\; V_{2}\; \ldots\; V_{s})
\end{equation}
where $\alpha_{1},\alpha_{2},\ldots,\alpha_{s}$ are distinct elements of $\mathbb{C}$.\\
The block matrix $V_{k}$ is of order $n\times m_{k}, k=1,\ldots,s$, and defined to be the matrix
\begin{equation*}
V_{k}=V_{G}((x-\alpha_{k})^{m_{k}})
\end{equation*}
 with entries
\[  (V_{k})_{ij}= \left\{ \begin{array}{ll}
         \binom{i-1}{j-1}\alpha_{k}^{i-j} & \mbox{if $i \geq j$}\\
        0, & \mbox{otherwise },\end{array} \right. \]
where $\binom{q}{p}$ denotes the binomial coefficient.\\
\indent For completeness, we recall the following Theorem and corollary (needed in the sequel) which provide an explicit closed-form for the inverse of the confluent Vandermonde matrices (for more details, see~\cite{Mmous}).
\begin{theorem}\label{Thm 1.2}
Let $P(x)=(x-\alpha_{1})^{m_{1}}(x-\alpha_{2})^{m_{2}}\cdots(x-\alpha_{s})^{m_{s}}$ be a polynomial of degree $n$, where $\alpha_{1}, \alpha_{2},\ldots,\alpha_{s}$ are distinct elements of $\mathbb{C}$. The explicit inverse of the confluent Vandermonde matrix $V_{G}(P)$ has the form
\begin{equation}\label{eq: 1.7}
V_{G}^{-1}(P) = \begin{pmatrix}
   \mathcal{L}_{1{m_{1}}} \\
   \mathcal{L}_{2{m_{2}}}  \\
    \vdots  \\
    \mathcal{L}_{sm_{s}}  \\

\end{pmatrix}
\end{equation}
for $r=1,2,\ldots,s$ the block matrix $\mathcal{L}_{rm_{r}}$ is of order $m_{r}\times n$ and given by
\begin{equation*}
\mathcal{L}_{rm_{r}}=\left(\frac{1}{(j-1)!}L^{(j-1)}_{r(i-1)}(0)[P]\right)_{1\leq i\leq m_{r},1\leq j\leq n}
\end{equation*}
More precisely,
\begin{equation*}
\mathcal{L}_{rm_{r}}= \begin{pmatrix}

  L_{r0}(0)[P] & L^{(1)}_{r0}(0)[P] & \cdots & \frac{1}{(n-1)!}L^{(n-1)}_{r0}(0)[P]  \\
  L_{r1}(0)[P] & L^{(1)}_{r1}(0)[P] & \cdots & \frac{1}{(n-1)!}L^{(n-1)}_{r1}(0)[P]  \\
    \vdots & \vdots & \cdots & \vdots \\
    L_{rm_{r}-1}(0)[P] & L^{(1)}_{rm_{r}-1}(0)[P] & \cdots & \frac{1}{(n-1)!}L^{(n-1)}_{rm_{r}-1}(0)[P]  \\
\end{pmatrix}
\end{equation*}
\end{theorem}
The following corollary treats the interesting case of the generalized Pascal matrix
\begin{corollary}\label{Prop 1.1}
Let $\alpha$ be a complex number and $n$ be a positive integer. Then the inverse of the Vandermonde matrix
\begin{equation}\label{eq: 1.5}
V_{G}((x-\alpha)^{n}) = \begin{pmatrix}
    1 &  0 &\cdots & 0   \\
    \alpha & 1 & \cdots &  0 \\
     \alpha^{2} &2 \alpha & \cdots & 0 \\
      \alpha^{3} &3 \alpha^{2} & \cdots & 0 \\
     \binom{4}{0} \alpha^{4} &\binom{4}{1} \alpha^{3} & \cdots & 0 \\
    \vdots & \vdots & \cdots & \vdots \\
    \binom{n-1}{0} \alpha^{n-1} & \binom{n-1}{1} \alpha^{n-2}  & \cdots & 1
\end{pmatrix}
\end{equation}
is
\begin{equation*}
 V_{G}^{-1}((x-\alpha)^{n}) = \begin{pmatrix}
    1 &  0 &\cdots & 0   \\
   -\alpha & 1 & \cdots &  0 \\
     \alpha^{2} &-2\alpha & \cdots & 0 \\
     -\alpha^{3} &3\alpha^{2} & \cdots & 0 \\
   \binom{4}{0}(-\alpha)^{4} & \binom{4}{1}(-\alpha)^{3}  & \cdots & 0 \\
    \vdots & \vdots & \cdots & \vdots \\
    \binom{n-1}{0}(-\alpha)^{n-1} & \binom{n-1}{1}(-\alpha)^{n-2}   & \cdots & 1
\end{pmatrix}
\end{equation*}
\end{corollary}
\section{Explicit formulas for the matrix exponential}
\label{sec:2}
Let us consider the following system of differential equations
\begin{equation*}
\dot{X}=AX
\end{equation*}
where $A=(a_{ij})_{1\leq i,j\leq k}$ is a constant matrix with entries in $\mathbb{C}$ and $X(t)$ the vector column defined by $X(t)=(x_{1}(t),x_{2}(t),\ldots,x_{k}(t))^{\mathit{T}}$.\\
Given any square matrix $A$, the exponential matrix function is
\begin{equation*}
\e^{tA}=\sum_{n=0}^{+\infty}\frac{t^{n}}{n!}A^{n}
\end{equation*}
It is well known that the function
\begin{equation*}
X(t)=\e^{tA}X_{0}
\end{equation*}
is the theoretical solution of the equation
\begin{equation*}
\dot{X}=AX, \quad X(0)=X_{0}
\end{equation*}
The last differential equation has gained much importance in linear and dynamical systems.\\
\indent Using the results of the previous section, we develop a purely algebraic approach, that requires only the knowledge of eigenvalues of the matrix, to derive more explicit expressions for the exponential of an arbitrary complex matrix.\\
\indent Let $\chi$ be a unital polynomial of degree $k$
\begin{equation*}
\chi(x)=x^{k}-a_{1}x^{k-1}-a_{2}x^{k-2}-\cdots-a_{k}
\end{equation*}
and consider the following set of differentiable functions mapping $\mathbb{C}$ to the complex matrices of order $k$
\begin{equation*}
F(\chi)=\{f/ f^{(k)}(t)=a_{1}f^{(k-1)}(t)+a_{2}f^{(k-2)}(t)+\cdots+a_{k}f(t)\}
\end{equation*}
where $f^{(l)}(t)$ denote the $l$th derivative of $f(t)$. It is well known that $F(\chi)$ is a $\mathbb{C}$-vector space of dimension $k$.\\
Let $D(\chi)$ be the vector space of all complex functions satisfying the following linear differential equation
\begin{equation*}
y^{(k)}(t)=a_{1}y^{(k-1)}(t)+a_{2}y^{(k-2)}(t)+\cdots+a_{k}y(t)
\end{equation*}
Let $y_{i}, 0\leq i\leq k-1$, be the elements of $D(\chi)$ with the initial conditions
\begin{equation*}
y^{(j)}_{i}(0)=\delta_{ij},\,\,\text{for}\,\, 0\leq j\leq k-1.
\end{equation*}
It is well known that $\lbrace y_{0}(t),y_{1}(t),\ldots,y_{k-1}(t) \rbrace$ is a basis of $D(\chi)$; it is called the canonical basis of $D(\chi)$. The most important property of this basis is that all the elements of $F(\chi)$ can be expressed as the linear combinations of the $y_{i}$'s only in terms of their derivatives at $t=0$. More precisely, each $f\in F(\chi)$ can be written
\begin{equation}\label{eq: sss}
f(t)=\sum_{i=0}^{k-1}y_{i}(t)f^{(i)}(0)
\end{equation}
As a particular case of Formula~\eqref{eq: sss}, we find the following well known result.
\begin{proposition}\label{Thm 2.1}
For any $k\times k$ matrix $A$, the exponential matrix function is given by
\begin{equation*}
\e^{tA}=\sum_{i=0}^{k-1}y_{i}(t)A^{i}
\end{equation*}
\end{proposition}
\begin{proof} Follows immediately from the fact that $\e^{tA}\in F(\chi_{A})$, where $\chi_{A}$ is the characteristic polynomial of $A$ and $\lbrace y_{0}(t),y_{1}(t),\ldots,y_{k-1}(t) \rbrace$ is the canonical basis of $D(\chi_{A})$.
\end{proof}
In~\cite{Tmapo} the author intents to find, in the simplest way, a method for computing the exponential of a matrix but he does not treat all cases. In what follows, we present the general results without any disadvantages.
We begin with a proposition that will be useful in the sequel.
\begin{proposition}\label{Prop 2.2}
If $\lbrace e_{1}(t),e_{2}(t),\ldots,e_{k}(t) \rbrace $ is a basis of the vector space $D(\chi_{A})$, then there exist unique constant matrices $B_{1},B_{2},\ldots,B_{k}$ such that
\begin{equation*}
\e^{tA}=e_{1}(t)B_{1}+e_{2}(t)B_{2}+\cdots+e_{k}(t)B_{k}
\end{equation*}
and consequently, we have
\begin{equation*}
A^{n}=e_{1}^{(n)}(0)B_{1}+e_{2}^{(n)}(0)B_{2}+\cdots+e_{k}^{(n)}(0)B_{k}
\end{equation*}
for every $n\in \mathbb{N}$.
\end{proposition}
\begin{proof}
This result follows directly from Proposition~\ref{Thm 2.1} and the fact that $\lbrace e_{1}(t),e_{2}(t),\ldots,e_{k}(t) \rbrace$ is a basis of the vector space $D(\chi_{A})$.
\end{proof}
In the following, we provide the explicit expressions of the elements $y_{0}(t),y_1(t),\ldots,y_{k-1}(t)$ of $D(\chi)$ in terms of the roots of $\chi$.
\begin{theorem}\label{Thm 2.4}
Let $\chi(x)=(x-\alpha_{1})^{m_{1}}(x-\alpha_{2})^{m_{2}}\cdots(x-\alpha_{s})^{m_{s}}$ be a unital polynomial of degree $k$. Then the (i+1)th element of the canonical basis of $D(\chi)$ is
\begin{equation*}
y_{i}(t)=\sum_{p=1}^{s}\frac{\e^{\alpha_{p}t}}{i!}(\sum_{r_{p}=0}^{m_{p}-1}\frac{t^{r_{p}}}{r_{p}!}L_{pr_{p}}^{(i)}(0)[\chi]), 0\leq i\leq k-1
\end{equation*}
\end{theorem}
\begin{proof}
It is well known that
$$B=\lbrace \e^{\alpha_{1}t},t\e^{\alpha_{1}t},\ldots,\frac{t^{m_{1}-1}}{(m_{1}-1)!}\e^{\alpha_{1}t},\ldots,\e^{\alpha_{s}t},t\e^{\alpha_{s}t},\ldots, \frac{t^{m_{s}-1}}{(m_{s}-1)!}\e^{\alpha_{s}t} \rbrace $$
is a basis of $D(\chi)$. By expressing each member of this basis in terms of the canonical basis, we get
\begin{equation*}
\frac{t^{i}}{i!}\e^{\alpha_{p}t}=\sum_{j=0}^{k-1}(\frac{t^{i}}{i!}\e^{\alpha_{p}t})^{(j)}(0)y_{j}(t)
\end{equation*}
for $0\leq i\leq m_{p}-1$ and $p=1,2,\ldots,s$. The resulting change of basis matrix from $B$ to the canonical basis is the confluent Vandermonde matrix~\eqref{eq: 1.4}
$$
V_{G}(\chi)=\begin{pmatrix}
    1 &  0 &\cdots & 0 & \cdots & 1 & 0 &\cdots&0  \\
    \alpha_{1} & 1 & \cdots &  0 & \cdots &  \alpha_{s}& 1 &\cdots& 0  \\
     \alpha_{1}^{2} & 2\alpha_{1} & \cdots & 0& \cdots &  \alpha_{s}^{2}& 2\alpha_{s} & \cdots& 0  \\
      \alpha_{1}^{3} &3\alpha_{1}^{2} & \cdots & 0 & \cdots &  \alpha_{s}^{3}& 3\alpha_{s}^{2} & \cdots&0  \\
    \vdots & \vdots & \cdots & \vdots & \cdots & \vdots & \vdots & \cdots & \vdots\\
     \alpha_{1}^{k-1} & (k-1)\alpha_{1}^{k-2}   & \cdots & 1 &\cdots& \alpha_{s}^{k-1}& (k-1)\alpha_{s}^{k-2} &\cdots& 1
\end{pmatrix}
$$
Using the inverse given by~\eqref{eq: 1.7}, we obtain
$$y_{i}(t)=\sum_{p=1}^{s}\frac{\e^{\alpha_{p}t}}{i!}(\sum_{r_{p}=0}^{m_{p}-1}\frac{t^{r_{p}}}{r_{p}!}L_{pr_{p}}^{(i)}(0)[\chi]), 0\leq i\leq k-1 $$
as desired.
\end{proof}
As a particular case we have the following corollary.
\begin{corollary}\label{Thm 2.3}
If $\chi(x)=(x-\alpha_{1})(x-\alpha_{2})\cdots(x-\alpha_{k})$ has distinct roots $\alpha_{1},\alpha_{2},\ldots,\alpha_{k}$, then
\begin{equation*}
y_{i}(t)=\sum_{j=1}^{k}\frac{\e^{\alpha_{j}t}}{i!}L_{j0}^{(i)}(0)[\chi], 0\leq i\leq k-1
\end{equation*}
\end{corollary}
In the following we state one of our main results.
\begin{theorem}\label{Thm 22.4}
Let $A$ be a $k\times k$ matrix, and let $\chi_{A}(x)=(x-\alpha_{1})^{m_{1}}(x-\alpha_{2})^{m_{2}}\cdots(x-\alpha_{s})^{m_{s}}$ be its characteristic polynomial. Then
\begin{equation*}
\e^{tA}=\sum_{i=0}^{k-1}[\sum_{p=1}^{s}\frac{\e^{\alpha_{p}t}}{i!}(\sum_{r_{p}=0}^{m_{p}-1}\frac{t^{r_{p}}}{r_{p}!}L_{pr_{p}}^{(i)}(0)[\chi_{A}])]A^{i}
\end{equation*}
\end{theorem}
\begin{proof}
The result is an immediate consequence of Theorem~\ref{Thm 2.4} and Proposition~\ref{Thm 2.1}.
\end{proof}
As a particular case we have the following corollary.
\begin{corollary}\label{Thm 2.3}
If $\chi_{A}(x)=(x-\alpha_{1})(x-\alpha_{2})\cdots(x-\alpha_{k})$ has distinct roots $\alpha_{1},\alpha_{2},\ldots,\alpha_{k}$, then
\begin{equation*}
\e^{tA}=\sum_{i=0}^{k-1}(\sum_{j=1}^{k}\frac{\e^{\alpha_{j}t}}{i!}L_{j0}^{(i)}(0)[\chi_{A}])A^{i}
\end{equation*}
\end{corollary}
\begin{proof}
Follows immediately from Theorem~\ref{Thm 2.4} and Theorem~\ref{Thm 22.4}.
\end{proof}
In the case where the matrix $A$ has a single eigenvalue, we have the following result.
\begin{corollary}\label{Thm 2.5}
If $\chi_{A}(x)=(x-\alpha)^{k}$ has a single root $\alpha$, then
\begin{equation*}
y_{i}(t)=\frac{\e^{\alpha t}}{i!}(\sum_{l=i}^{k-1}\frac{(-1)^{l+i}}{(l-i)!}\alpha^{l-i}t^{l}), 0\leq i\leq k-1
\end{equation*}
and
\begin{equation*}
\e^{tA}=\sum_{i=0}^{k-1}\frac{1}{i!}(\sum_{l=i}^{k-1}\frac{(-1)^{l+i}}{(l-i)!}\alpha^{l-i}t^{l})A^{i}\e^{\alpha t}
\end{equation*}
\end{corollary}
\begin{proof}
The result is an immediate consequence of Theorem~\ref{Thm 2.4} and Theorem~\ref{Thm 22.4}.
\end{proof}
For illustration purposes, we consider the case of a square matrix of order $3$.
\begin{example}
Let $A$ be a square matrix of order $3$ with characteristic polynomial
\begin{equation*}
\chi_{A}(x)=x^{3}-a_{1}x^{2}-a_{2}x-a_{3}
\end{equation*}
\begin{enumerate}[1.]
\item If $\alpha$ is a root of multiplicity $3$ of $\chi_{A}(x)$ then, using Theorem~\ref{Thm 2.5}, we have
$$\e^{tA}=(\e^{\alpha t}\sum_{l=0}^{2}\frac{(-1)^{l}}{l!}\alpha^{l}t^{l})I+(\e^{\alpha t}\sum_{l=1}^{2}\frac{(-1)^{l+1}}{(l-1)!}  \alpha^{l-1}t^{l})A+(\frac{\e^{\alpha t}}{2!}\sum_{l=2}^{2}\frac{(-1)^{l}}{(l-2)!}\alpha^{l-2}t^{l})A^{2}
$$
Consequently, we obtain
\begin{equation*}
\e^{tA}=(1-\alpha t+\frac{1}{2}\alpha^{2}t^{2})\e^{\alpha t}I+(t-\alpha t^{2})\e^{\alpha t}A+\frac{1}{2}t^{2}\e^{ \alpha t}A^{2}
\end{equation*}
\item If $\alpha_{1}$ is a root of multiplicity $2$, and $\alpha_{2}$ is a simple root, then Theorem~\ref{Thm 22.4} gives
\begin{align*}
     \e^{tA}&=\\
     &(\e^{\alpha_{1}t}(L_{10}(0)[\chi_{A}])+tL_{11}(0)[\chi_{A}])+\e^{\alpha_{2}t}L_{20}(0)[\chi_{A}])I+\\
    &\hspace*{1cm}(\e^{\alpha_{1}t}(L_{10}^{(1)}(0)[\chi_{A}])+tL_{11}^{(1)}(0)[\chi_{A}])+\e^{\alpha_{2}t}L_{20}^{(1)}(0)[\chi_{A}])A+\\
    &\hspace*{2cm}\frac{1}{2!} (\e^{\alpha_{1}t}(L_{10}^{(2)}(0)[\chi_{A}])+tL_{11}^{(2)}(0)[\chi_{A}])+\e^{\alpha_{2}t}L_{20}^{(2)}(0)[\chi_{A}])A^{2}
    \end{align*}
\begin{equation*}
\end{equation*}
Formula~\eqref{eq: 1.1} yields
$$ \left \{
\begin{array}{rcl}
L_{10}(x)[\chi_{A}]&=&\dfrac{-x^{2}+2\alpha_{1}x-2\alpha_{1}\alpha_{2}+\alpha_{2}^{2}}{(\alpha_{1}-\alpha_{2})^{2}}\vspace*{0.5pc}\\
L_{11}(x)[\chi_{A}]&=&\dfrac{x^{2}-(\alpha_{1}+\alpha_{2})x+\alpha_{1}\alpha_{2}}{(\alpha_{1}-\alpha_{2})}\vspace*{0.5pc}\\
L_{20}(x)[\chi_{A}]&=&\dfrac{x^{2}-2\alpha_{1}x+\alpha_{1}^{2}}{(\alpha_{2}-\alpha_{1})^{2}}\\
\end{array}
\right.
$$
A simple calculation gives
\begin{align*}
     \e^{tA}=&\frac{1}{(\alpha_{1}-\alpha_{2})^2}\big[((\alpha_{2}^2-2\alpha_{1}\alpha_{2})\e^{\alpha_{1}t}+(\alpha_{1}^{2}\alpha_{2}- \alpha_{1}\alpha_{2}^2)t\e^{\alpha_{1}t}+\alpha_{1}^2\e^{\alpha_{2}t})I+\\
    &((\alpha_{1}^{2}-\alpha_{2}^2)t\e^{\alpha_{1} t}+2\alpha_{1}(\e^{\alpha_{1}t}-\e^{\alpha_{2}t}))A+
    (\e^{\alpha_{2}t}-\e^{\alpha_{1} t}+(\alpha_{1}-\alpha_{2})t\e^{\alpha_{1} t})A^{2}\big]\\
    \end{align*}
\item If $\alpha_{1},\alpha_{2},\alpha_{3}$ are simple roots of $\chi_{A}(x)$ then, using Corollary~\ref{Thm 2.3}, we have
$$
\e^{tA}=(\sum_{j=1}^{3}\e^{\alpha_{j}t}L_{j0}(0)[\chi_{A}]))I+(\sum_{j=1}^{3}\e^{\alpha_{j}t}L_{j0}^{(1)}(0)[\chi_{A}]))A+
(\sum_{j=1}^{3}\frac{ \e^{\alpha_{j}t}}{2!}L_{j0}^{(2)}(0)[\chi_{A}]))A^{2}
$$
Formula~\eqref{eq: 1.1} gives
$$ \left \{
\begin{array}{rcl}
L_{10}(x)[\chi_{A}])&=&\dfrac{(x-\alpha_{2})(x-\alpha_{3})}{(\alpha_{1}-\alpha_{2})(\alpha_{1}-\alpha_{3})}=\dfrac{x^{2}-(\alpha_{2}+\alpha_{3})x+
\alpha_{2}\alpha_{3}}{(\alpha_{1}-\alpha_{2})(\alpha_{1}-\alpha_{3})}\vspace*{0.5pc}\\
L_{20}(x)[\chi_{A}])&=&\dfrac{(x-\alpha_{1})(x-\alpha_{3})}{(\alpha_{2}-\alpha_{1})(\alpha_{2}-\alpha_{3})}=\dfrac{x^{2}-(\alpha_{1}+\alpha_{3})x+
\alpha_{1}\alpha_{3}}{(\alpha_{2}-\alpha_{1})(\alpha_{2}-\alpha_{3})}\vspace*{0.5pc}\\
L_{30}(x)[\chi_{A}])&=&\dfrac{(x-\alpha_{1})(x-\alpha_{2})}{(\alpha_{3}-\alpha_{1})(\alpha_{3}-\alpha_{2})}=\dfrac{x^{2}-(\alpha_{1}+\alpha_{2})x+
\alpha_{1}\alpha_{2}}{(\alpha_{3}-\alpha_{1})(\alpha_{3}-\alpha_{2})}
\end{array}
\right.
$$
Consequently, the matrix exponential in this case is given by
\begin{align*}
     &\e^{tA}=\\
     &\left(\frac{\alpha_{2}\alpha_{3}}{(\alpha_{1}-\alpha_{2})(\alpha_{1}-\alpha_{3})}\e^{\alpha_{1} t}+\frac{\alpha_{1}\alpha_{3}}{(\alpha_{2}-\alpha_{1})(\alpha_{2}-\alpha_{3})}\e^{\alpha_{2}t}+
     \frac{\alpha_{1}\alpha_{2}}{(\alpha_{3}-\alpha_{1})(\alpha_{3}-\alpha_{2})}\e^{\alpha_{3}t}\right)I\\
    &+\left(\frac{\alpha_{2}+\alpha_{3}}{(\alpha_{3}-\alpha_{1})(\alpha_{1}-\alpha_{2})}\e^{\alpha_{1} t}+\frac{\alpha_{1}+\alpha_{3}}{(\alpha_{3}-\alpha_{2})(\alpha_{2}-\alpha_{1})}\e^{\alpha_{2}t}+
    \frac{\alpha_{1}+\alpha_{2}}{(\alpha_{1}-\alpha_{3})(\alpha_{3}-\alpha_{2})}\e^{\alpha_{3}t}\right)A\\
    &\hspace*{1cm}+\left(\frac{\e^{\alpha_{1} t}}{(\alpha_{1}-\alpha_{2})(\alpha_{1}-\alpha_{3})}+\frac{\e^{\alpha_{2}t}}{(\alpha_{2}-\alpha_{1})(\alpha_{2}-\alpha_{3})}+
    \frac{\e^{\alpha_{3}t}}{(\alpha_{3}-\alpha_{1})(\alpha_{3}-\alpha_{2})}\right)A^{2}
    \end{align*}
\end{enumerate}
\end{example}
In the present approach, one gains a basic construction of the so-called spectral decomposition of $A$. The approach makes the determination of the spectral decomposition of any square matrix more practical than the usual method of partial fraction decomposition. Using Lagrange polynomials W. A. Harris et al.~\cite{Harri} have derived the spectral decomposition of a matrix with simple eigenvalues. Here, we generalize this result to any matrix using a generalization of Hermite's formula given by A.~Spitzbart~\cite{As1}.
\begin{theorem}\label{Thm 2.7}
Let $A$ be a $k\times k$ matrix, and let $\chi_{A}(x)=(x-\alpha_{1})^{m_{1}}(x-\alpha_{2})^{m_{2}}\cdots(x-\alpha_{s})^{m_{s}}$ be its characteristic polynomial. Then
\begin{equation}\label{eq: 2.20}
\e^{tA}=\sum_{j=1}^{s}\sum_{k_{j}=0}^{m_{j}-1}\frac{t^{k_{j}}}{k_{j}!}\e^{\alpha_{j}t}B_{jk_{j}}
\end{equation}
where $B_{jk_{j}}=L_{jk_{j}}(A)[\chi_{A}]$. Moreover, for $i\neq j$ we have
\begin{equation}\label{eq: 2.21}
B_{ik_{i}}B_{jk_{j}}=B_{jk_{j}}B_{ik_{i}}=0
\end{equation}
\end{theorem}
\begin{proof}
Using Proposition~\ref{Prop 2.2} with the fact that $$\lbrace \e^{\alpha_{1}t},t\e^{\alpha_{1}t},\ldots,\frac{t^{m_{1}-1}}{(m_{1}-1)!}\e^{\alpha_{1}t},\ldots,\e^{\alpha_{s}t},t\e^{\alpha_{s}t},\ldots, \frac{t^{m_{s}-1}}{(m_{s}-1)!}\e^{\alpha_{s}t} \rbrace $$ is a basis of $D(\chi_{A})$, we can find unique constant matrices $B_{jk_{j}}, j=1,2,\ldots,s$ and $0\leq k_{j}\leq m_{j}-1$, such that
$$\e^{tA}=\sum_{j=0}^{m_{1}-1}\frac{t^{j}}{j!}\e^{\alpha_{1}t}B_{1j}+\sum_{j=0}^{m_{2}-1}\frac{t^{j}}{j!}\e^{\alpha_{2}t}B_{2j}+\cdots+\sum_{j=0}^{m_{s}-1}\frac{t^{j}}{j!}\e^{\alpha_{s}t}B_{sj}$$
Applying Proposition~\ref{Prop 2.2} to this equality yields
\begin{align}\label{eq: xyz}
&\begin{pmatrix}\nonumber
    1 &  0 &\cdots & 0 & \cdots & 1 & 0 &\cdots&0  \\
   \alpha_{1} & 1 & \cdots &  0 & \cdots & \alpha_{s}& 1 &\cdots& 0  \\
    \alpha_{1}^{2} & 2\alpha_{1} & \cdots & 0& \cdots & \alpha_{s}^{2}& 2\alpha_{s} & \cdots& 0  \\
     \alpha_{1}^{3} &3\alpha_{1}^{2} & \cdots & 0 & \cdots & \alpha_{s}^{3}& 3\alpha_{s}^{2} & \cdots&0  \\
    \vdots & \vdots & \cdots & \vdots & \cdots & \vdots & \vdots & \cdots & \vdots\\
    \vdots & \vdots & \cdots & \vdots & \cdots & \vdots & \vdots & \cdots & \vdots\\
    \vdots & \vdots & \cdots & \vdots & \cdots & \vdots & \vdots & \cdots & \vdots\\
    \alpha_{1}^{k-1} & (k-1)\alpha_{1}^{k-2}   & \cdots & 1 &\cdots&\alpha_{s}^{k-1}& (k-1)\alpha_{s}^{k-2} &\cdots& 1
\end{pmatrix}\\
&\hspace*{7cm}\begin{pmatrix}
B_{10}\\
B_{11} \\
\vdots \\
B_{1m_{1}-1}\\
\vdots \\
B_{s0}\\
\vdots \\
B_{sm_{s}-1}
\end{pmatrix}
=
\begin{pmatrix}
I\\
A\\
\vdots \\
A^{m_{1}-1} \\
\vdots \\
\vdots \\
\vdots \\
A^{k-1}
\end{pmatrix}
\end{align}
On the other hand, utilizing Formula~\eqref{eq: 1.3} for the canonical basis of $\mathbb{C}_{k-1}[x]$, we obtain the following system
\begin{align*}
&\begin{pmatrix}
    1 &  0 &\cdots & 0 & \cdots & 1 & 0 &\cdots&0  \\
   \alpha_{1} & 1 & \cdots &  0 & \cdots & \alpha_{s}& 1 &\cdots& 0  \\
    \alpha_{1}^{2} & 2\alpha_{1} & \cdots & 0& \cdots & \alpha_{s}^{2}& 2\alpha_{s} & \cdots& 0  \\
     \alpha_{1}^{3} &3\alpha_{1}^{2} & \cdots & 0 & \cdots & \alpha_{s}^{3}& 3\alpha_{s}^{2} & \cdots&0  \\
    \vdots & \vdots & \cdots & \vdots & \cdots & \vdots & \vdots & \cdots & \vdots\\
    \vdots & \vdots & \cdots & \vdots & \cdots & \vdots & \vdots & \cdots & \vdots\\
    \vdots & \vdots & \cdots & \vdots & \cdots & \vdots & \vdots & \cdots & \vdots\\
    \alpha_{1}^{k-1} & (k-1)\alpha_{1}^{k-2}   & \cdots & 1 &\cdots&\alpha_{s}^{k-1}& (k-1)\alpha_{s}^{k-2} &\cdots& 1
\end{pmatrix}
\\&\hspace*{7cm}
\begin{pmatrix}
L_{10}(x)[\chi_{A}]\\
L_{11}(x)[\chi_{A}] \\
\vdots \\
L_{1m_{1}-1}(x)[\chi_{A}]\\
\vdots \\
L_{s0}(x)[\chi_{A}]\\
\vdots \\
L_{sm_{s}-1}(x)[\chi_{A}]
\end{pmatrix}
=
\begin{pmatrix}
1\\
x\\
\vdots \\
x^{m_{1}-1} \\
\vdots \\
\vdots \\
\vdots \\
x^{k-1}
\end{pmatrix}
\end{align*}
Replacing in this matrix equation $x$ with $A$ yields
\begin{align*}
&\begin{pmatrix}
    1 &  0 &\cdots & 0 & \cdots & 1 & 0 &\cdots&0  \\
   \alpha_{1} & 1 & \cdots &  0 & \cdots & \alpha_{s}& 1 &\cdots& 0  \\
    \alpha_{1}^{2} & 2\alpha_{1} & \cdots & 0& \cdots & \alpha_{s}^{2}& 2\alpha_{s} & \cdots& 0  \\
     \alpha_{1}^{3} &3\alpha_{1}^{2} & \cdots & 0 & \cdots & \alpha_{s}^{3}& 3\alpha_{s}^{2} & \cdots&0  \\
    \vdots & \vdots & \cdots & \vdots & \cdots & \vdots & \vdots & \cdots & \vdots\\
    \vdots & \vdots & \cdots & \vdots & \cdots & \vdots & \vdots & \cdots & \vdots\\
    \vdots & \vdots & \cdots & \vdots & \cdots & \vdots & \vdots & \cdots & \vdots\\
    \alpha_{1}^{k-1} & (k-1)\alpha_{1}^{k-2}   & \cdots & 1 &\cdots&\alpha_{s}^{k-1}& (k-1)\alpha_{s}^{k-2} &\cdots& 1
\end{pmatrix}\\
&\hspace*{7.5cm}\begin{pmatrix}
L_{10}(A)[\chi_{A}]\\
L_{11}(A)[\chi_{A}] \\
\vdots \\
L_{1m_{1}-1}(A)[\chi_{A}]\\
\vdots \\
L_{s0}(A)[\chi_{A}]\\
\vdots \\
L_{sm_{s}-1}(A)[\chi_{A}]
\end{pmatrix}
=
\begin{pmatrix}
I\\
A\\
\vdots \\
A^{m_{1}-1} \\
\vdots \\
\vdots \\
\vdots \\
A^{k-1}
\end{pmatrix}
\end{align*}
Since the confluent Vandermonde matrix is invertible, we have
$$B_{jk_{j}}=L_{jk_{j}}(A)[\chi_{A}] \,\,\text{for all}\,\, 1\leq j \leq s,\,\,0\leq k_{j} \leq m_{j}-1$$
On the other hand, by using~\eqref{eq: 1.1}, we can show that $\chi_{A}(x)$ divides
\begin{equation*}
L_{jk_{j}}(x)[\chi_{A}]L_{ik_{i}}(x)[\chi_{A}]
\end{equation*}
for $i\neq j$. As a consequence, we get
\begin{equation*}
B_{jk_{j}}B_{ik_{i}}=B_{ik_{i}}B_{jk_{j}}=0
\end{equation*}
\end{proof}
The following example is used to illustrate Theorem~\ref{Thm 22.4} and Theorem~\ref{Thm 2.7}.
\begin{example}\label{ex 2.12}
Let us consider Example $1.$ of~\cite{Eileo}
\begin{equation*}
A=\begin{pmatrix}
2 &  0 & 1   \\
   0 & 2 &0 \\
   0 & 0 & 3
\end{pmatrix}
\end{equation*}
The characteristic polynomial of $A$ is $\chi_{A}(x)=(x-2)^{2}(x-3)$. In this case, by using Formula~\eqref{eq: 1.1}, we have
\begin{align*}
L_{10}(x)[\chi_{A}]&=(x-3)(1-x)=-x^{2}+4x-3\\
L_{11}(x)[\chi_{A}]&=(x-3)(2-x)=-x^{2}+5x-6\\
L_{20}(x)[\chi_{A}]&=(x-2)^{2}=x^{2}-4x+4\\
\end{align*}
We deduce that
\begin{equation*}
\begin{array}{ccccccccc}
L_{10}(0)[\chi_{A}]&=&-3, &L_{10}^{'}(0)[\chi_{A}]&=&4, &L_{10}^{''}(0)[\chi_{A}]&=&-2 \\
L_{11}(x)[\chi_{A}]&=&-6, &L_{11}^{'}(0)[\chi_{A}]&=&5, &L_{11}^{''}(0)[\chi_{A}]&=&-2\\
L_{20}(x)[\chi_{A}]&=&4, &L_{20}^{'}(0)[\chi_{A}]&=&-4, &L_{20}^{''}(0)[\chi_{A}]&=&2\\
\end{array}
\end{equation*}
Using the formula of Theorem~\ref{Thm 22.4}, we obtain
$$
\e^{tA}=(\e^{2t}(-3-6t)+4\e^{3t})I+(\e^{2t}(4+5t)-4\e^{3t})A+
\frac{1}{2}(\e^{2t}(-2-2t)+2\e^{3t})A^{2}
$$
Therefore
\begin{equation*}
\e^{tA}=\begin{pmatrix}
\e^{2t} &  0 & \e^{3t}-\e^{2t} \\
   0 & \e^{2t} &0 \\
   0 & 0 & \e^{3t}
\end{pmatrix}
\end{equation*}
Now we find the matrix exponential of $A$, but using this time Theorem~\ref{Thm 2.7}. From Formula~\eqref{eq: 2.20}, we get
\begin{equation*}
\e^{tA}=\e^{2t}B_{10}+t\e^{2t}B_{11}+\e^{3t}B_{20}
\end{equation*}
where
\begin{align*}
B_{10}&=L_{10}(A)[\chi_{A}]=-A^{2}+4A-3I\\
B_{11}&=L_{11}(A)[\chi_{A}]=-A^{2}+5A-6I\\
B_{20}&=L_{20}(A)[\chi_{A}]=A^{2}-4A+4I\\
\end{align*}
More explicitly,
$$
B_{10}=\begin{pmatrix}
1 &  0 & -1   \\
   0 & 1 &0 \\
   0 & 0 & 0
\end{pmatrix},
B_{11}=\begin{pmatrix}
0 &  0 & 0   \\
   0 & 0 &0 \\
   0 & 0 & 0
\end{pmatrix},
B_{20}=\begin{pmatrix}
0 &  0 & 1   \\
   0 & 0 &0 \\
   0 & 0 & 1
\end{pmatrix}
$$
Thus
\begin{equation*}
\e^{tA}=\e^{2t}B_{10}+t\e^{2t}B_{11}+\e^{3t}B_{20}
\end{equation*}
That is
\begin{equation*}
\e^{tA}=\begin{pmatrix}
\e^{2t} &  0 & \e^{3t}-\e^{2t} \\
   0 & \e^{2t} &0 \\
   0 & 0 & \e^{3t}
\end{pmatrix}
\end{equation*}
\end{example}
Next, we derive some corollaries of Theorem~\ref{Thm 2.7}.
\begin{corollary}
If $\chi_{A}(x)=(x-\alpha_{1})^{m_{1}}(x-\alpha_{2})^{m_{2}}$, $\alpha_{1}$ and $\alpha_{2}$ are two distinct complex numbers, then
\begin{equation*}
\e^{tA}=\sum_{j=0}^{m_{1}-1}\frac{t^{j}}{j!}\e^{\alpha_{1}t}B_{1j}+\sum_{j=0}^{m_{2}-1}\frac{t^{j}}{j!}\e^{\alpha_{2}t}B_{2j}
\end{equation*}
$$ \left \{
\begin{array}{rcl}
B_{1j}&=&(A-\alpha_{1}I)^{j}(A-\alpha_{2}I)^{m_{2}}\displaystyle \sum_{i=0}^{m_{1}-j-1}\frac{(-1)^{i}\binom{m_{2}+i-1}{m_{2}-1}}{(\alpha_{1}-\alpha_{2})^{m_{2}+i}}(A-\alpha_{1}I)^{i}\\
B_{2j}&=&(A-\alpha_{1}I)^{m_{1}}(A-\alpha_{2}I)^{j}\displaystyle \sum_{i=0}^{m_{2}-j-1}\frac{(-1)^{i}\binom{m_{1}+i-1}{m_{1}-1}}{(\alpha_{2}-\alpha_{1})^{m_{1}+i}}(A-\alpha_{2}I)^{i}
\end{array}
\right.
$$
\end{corollary}
\begin{proof}
This follows directly from~\eqref{eq: 2.20} and the fact that
$$ \left \{
\begin{array}{rcl}
L_{1j}(x)[\chi_{A}]&=&\displaystyle\sum_{i=0}^{m_{1}-j-1}\frac{(-1)^{i}\binom{m_{2}+i-1}{m_{2}-1}}{(\alpha_{1}-\alpha_{2})^{m_{2}+i}}(x-\alpha_{1})^{i+j}
(x-\alpha_{2})^{m_{2}}\\
L_{2j}(x)[\chi_{A}]&=&\displaystyle\sum_{i=0}^{m_{2}-j-1}\frac{(-1)^{i}\binom{m_{1}+i-1}{m_{1}-1}}{(\alpha_{2}-\alpha_{1})^{m_{1}+i}}(x-\alpha_{1})^{m_{1}}
(x-\alpha_{2})^{i+j}
\end{array}
\right.
$$
\end{proof}
\begin{corollary}
If $\chi_{A}(x)=(x-\alpha_{1})^{m_{1}}(x-\alpha_{2})$, $\alpha_{1}$ and $\alpha_{2}$ are distinct two complex numbers, then
\begin{equation*}
\e^{tA}=\sum_{j=0}^{m_{1}-1}\frac{t^{j}}{j!}\e^{\alpha_{1}t}B_{1j}+\e^{\alpha_{2}t}B_{20}
\end{equation*}
where
$$
\left \{
\begin{array}{rcl}
B_{1j}&=&(A-\alpha_{1}I)^{j}(A-\alpha_{2}I)\displaystyle\sum_{i=0}^{m_{1}-j-1}\frac{(-1)^{i}}{(\alpha_{1}-\alpha_{2})^{i+1}}(A-\alpha_{1}I)^{i}\\
B_{20}&=&\displaystyle\frac{1}{(\alpha_{2}-\alpha_{1})^{m_{1}}}(A-\alpha_{1}I)^{m_{1}}
\end{array}
\right.
$$
\end{corollary}
More generally, we have the following result.
\begin{corollary}\label{Corr102}
If $\chi_{A}(x)=(x-\alpha_{1})^{m_{1}}\prod\limits_{j=2}^{s}(x-\alpha_{j})$, $\alpha_{1},\ldots, \alpha_{s}$ are distinct complex numbers, then
\begin{equation}\label{eq: aaaa}
\e^{tA}=\sum_{j=0}^{m_{1}-1}\frac{t^{j}}{j!}\e^{\alpha_{1}t}B_{1j}+\sum_{j=2}^{s}\e^{\alpha_{j}t}\frac{1}{P_{j}(\alpha_{j})}P_{j}(A)
\end{equation}
where
$P=\chi_{A},$
$$B_{1j}=(A-\alpha_{1}I)^{j}\prod\limits_{l=2}^{s}(A-\alpha_{l}I)\sum_{i=0}^{m_{1}-j-1}\sum_{l=2}^{s}\frac{(-1)^{i}
a_{l}}{(\alpha_{1}-\alpha_{l})^{i+1}}(A-\alpha_{1}I)^{i}$$
and
\begin{align*}
a_{l}=\displaystyle\begin{cases}
\dfrac{1}{\prod\limits_{p=2,p\neq l}^{s}(\alpha_{l}-\alpha_{p})} &\text{if}\quad s\geq3\\
1 &\text{if}\quad s=2
\end{cases}
\end{align*}
\end{corollary}
\begin{proof}
In this case, Formula~\eqref{eq: 2.20} becomes
\begin{equation*}
\e^{tA}=\sum_{j=0}^{m_{1}-1}\frac{t^{j}}{j!}\e^{\alpha_{1}t}B_{1j}+\e^{\alpha_{2}t}B_{20}+\cdots+\e^{\alpha_{s}t}B_{s0}
\end{equation*}
where, for $j=2,\ldots,s$,
\begin{equation*}
B_{j0}=L_{j0}(A)[P]=P_{j}(A)g_{j}(\alpha_{j})=\frac{1}{P_{j}(\alpha_{j})}P_{j}(A)
\end{equation*}
On the other hand, for $j=0,\ldots,m_{1}-1$, we have
$B_{1j}=L_{1j}(A)[P]$ where
\begin{align*}
L_{1j}(x)[P]&=P_{1}(x)(x-\alpha_{1})^{j}\sum_{i=0}^{m_{1}-1-j}\frac{1}{i!}g^{(i)}_{1}(\alpha_{1})(x-\alpha_{1})^{i},\\
P_{1}(x)&=\prod_{l=2}^{s}(x-\alpha_{l})\\
g_{1}(x)&=\frac{1}{\prod\limits_{l=2}^{s}(x-\alpha_{l})}
\end{align*}
But since
\begin{equation*}
g_{1}(x)=\sum_{l=2}^{s} \frac{a_{l}}{x-\alpha_{l}}
\end{equation*}
we have
\begin{equation*}
g_{1}^{(i)}(x)=\sum_{l=2}^{s} \frac{(-1)^{i}i!a_{l}}{(x-\alpha_{l})^{i+1}}.
\end{equation*}
Then
\begin{equation*}
L_{1j}(x)[P]=(x-\alpha_{1})^{j}\prod\limits_{l=2}^{s}(x-\alpha_{l})\sum_{i=0}^{m_{1}-1-j}\sum_{l=2}^{s} \frac{(-1)^{i}a_{l}}{(\alpha_{1}-\alpha_{l})^{i+1}} (x-\alpha_{1})^{i}.
\end{equation*}
Therefore
\begin{equation*}
B_{1j}=(A-\alpha_{1}I)^{j}\prod_{l=2}^{s}(A-\alpha_{l}I)\sum_{i=0}^{m_{1}-j-1}\sum_{l=2}^{s}\frac{(-1)^{i}
a_{l}}{(\alpha_{1}-\alpha_{l})^{i+1}}(A-\alpha_{1}I)^{i}.
\end{equation*}
Thus, the proof is completed.
\end{proof}
To illustrate Corollary~\ref{Corr102}, consider the following example.
\begin{example}
Let $$A=\begin{pmatrix}
1&1&0&0\\1&1&0&0\\2&3&-1&1\\1&1&1&-1
\end{pmatrix}$$
The characteristic polynomial of $A$ is $\chi_{A}(x)=x^{2}(x+2)(x-2)$. Let us choose, for example, $\alpha_{1}=0,\alpha_{2}=-2$ and $\alpha_{3}=2$. Then,
in light of Formula~\eqref{eq: aaaa}, the exponential matrix of $A$ is given by
\begin{equation}\label{eqq101}
\e^{tA}=B_{10}+tB_{11}+\e^{-2t}B_{20}+\e^{2t}B_{30}
\end{equation}
where
\begin{align*}
\begin{cases}
B_{10}=\displaystyle(A+2I)(A-2I)\Big[(\frac{a_{2}}{\alpha_{1}-\alpha_{2}}+\frac{a_{3}}{\alpha_{1}-\alpha_{3}})I-
(\frac{a_{2}}{(\alpha_{1}-\alpha_{2})^{2}}+\frac{a_{3}}{(\alpha_{1}-\alpha_{3})^{2}})A\Big]\vspace*{0.5pc}\\
B_{11}=A(A+2I)(A-2I)\displaystyle\Big[(\frac{a_{2}}{\alpha_{1}-\alpha_{2}}+\frac{a_{3}}{\alpha_{1}-\alpha_{3}})\Big]
\end{cases}
\end{align*}
and
$$
\left \{
\begin{array}{rcl}
B_{20}&=&\displaystyle\frac{1}{P_{2}(\alpha_{2})}P_{2}(A)=\frac{-1}{16}A^{2}(A-2I)=\frac{-1}{16}A^{3}+\frac{1}{8}A^{2}\vspace*{0.5pc}\\
B_{30}&=&\displaystyle\frac{1}{P_{3}(\alpha_{3})}P_{3}(A)=\frac{1}{16}A^{2}(A+2I)=\frac{1}{16}A^{3}+\frac{1}{8}A^{2}
\end{array}
\right.
$$
In this case $a_{2}=\displaystyle\frac{1}{\alpha_{2}-\alpha_{3}}=\frac{-1}{4}$ and $a_{3}=\displaystyle\frac{1}{\alpha_{3}-\alpha_{2}}=\frac{1}{4}$. Consequently, we have
\begin{equation*}
\left \{
\begin{array}{rcl}
B_{10}&=&\dfrac{-1}{4}A^{2}+I\vspace*{0.2pc}\\
B_{11}&=&\dfrac{-1}{4}A^{3}+A\vspace*{0.2pc}\\
B_{20}&=&\dfrac{-1}{16}A^{3}+\dfrac{1}{8}A^{2}\vspace*{0.2pc}\\
B_{30}&=&\dfrac{1}{16}A^{3}+\dfrac{1}{8}A^{2}
\end{array}
\right.
\end{equation*}
Hence Formula~\eqref{eqq101} becomes
\begin{equation*}
\e^{tA}=(\frac{\e^{2t}}{16}-\frac{-\e^{-2t}}{16}-\frac{t}{4})A^{3}+(\frac{\e^{2t}}{8}+\frac{\e^{-2t}}{8}-\frac{1}{4})A^{2}+tA+I
\end{equation*}
Since $$A^{2}=\begin{pmatrix}
2&2&0&0\\2&2&0&0\\4&3&2&-2\\3&4&-2&2
\end{pmatrix}$$
and
$$A^{3}=\begin{pmatrix}
4&4&0&0\\4&4&0&0\\9&11&-4&4\\5&3&4&-4
\end{pmatrix}$$
we obtain
\renewcommand{\arraystretch}{2.2}
$$
\e^{tA}=\begin{pmatrix}
\dfrac{\e^{2t}+1}{2}&\dfrac{\e^{2t}-1}{2}&0&0\\
\dfrac{\e^{2t}-1}{2}&\dfrac{\e^{2t}+1}{2}&0&0\\
\dfrac{17\e^{2t}-\e^{-2t}-4t-16}{16}&\dfrac{17\e^{2t}-5\e^{-2t}+4t-12}{16}&\dfrac{\e^{-2t}+1}{2}&\dfrac{-\e^{-2t}+1}{2}\\
\dfrac{11\e^{2t}+\e^{-2t}-4t-12}{16}&\dfrac{11\e^{2t}+5\e^{-2t}+4t-16}{16}&\dfrac{-\e^{-2t}+1}{2}&\dfrac{\e^{-2t}+1}{2}
\end{pmatrix}
$$
\end{example}
The following result is due to W. A. Harris et al.~\cite{Harri}
\begin{corollary}\label{Thm 2.8}
If $\chi_{A}(x)=(x-\alpha_{1})(x-\alpha_{2})\cdots(x-\alpha_{k})$ has distinct roots $\alpha_{1},\alpha_{2},\ldots,\alpha_{k}$, then
\begin{equation*}
\e^{tA}=\e^{\alpha_{1}t}B_{1}+\e^{\alpha_{2}t}B_{2}+\cdots+\e^{\alpha_{k}t}B_{k}
\end{equation*}
where $B_{i}=\prod_{j=1,j\neq i}^{k}\dfrac{x-\alpha_{j}}{\alpha_{i}-\alpha_{j}}$. Moreover, $B_{i}B_{j}=B_{j}B_{i}=0$ if $i\neq j$, and $B_{i}^2=B_{i}$.
\end{corollary}
\begin{proof}
Is a particular case of Theorem~\ref{Thm 2.7}.
\end{proof}
When $A$ is a matrix with only one eigenvalue, we have the following known result shown by Apostol~\cite{Tmapo}
\begin{corollary}\label{Thm 2.9}
If $\chi_{A}(x)=(x-\alpha)^{k}$ has a single root $\alpha$, then
\begin{equation*}
\e^{tA}=\e^{\alpha t}B_{1}+t\e^{\alpha t}B_{2}+\cdots+\frac{t^{k-1}}{(k-1)!}\e^{\alpha t}B_{k}
\end{equation*}
where $B_{i}=(A-\alpha I)^{i-1},1\leq i\leq k$.
\end{corollary}
\begin{proof}
Follows immediately from Theorem~\ref{Thm 2.7}.
\end{proof}
\section{Closed-form formula for the $n$th power of a matrix and Drazin inverses}
\label{sec:3}
In this section, we present some applications of the previous results. We begin by proposing a new closed-form formula for the $n$th power of an arbitrary complex matrix $A$. We then deduce another closed-form formula for the $n$th power of the Drazin inverse $A_{D}$ of $A$. Also, we will be interested in determining explicit formulas for the Chevalley--Jordan decomposition and the spectral projections of $A$. Our contribution does not consist only in giving these formulas, but also in making their determination much practical and expressing them in an elegant representations.
\begin{theorem}\label{Thm 3.1}
Let $A$ be a $k\times k$ matrix, and let $\chi_{A}(x)=(x-\alpha_{1})^{m_{1}}(x-\alpha_{2})^{m_{2}}\cdots(x-\alpha_{s})^{m_{s}}$,$\alpha_{1}=0$, be its characteristic polynomial (possibly $m_{1}=0$). Then for every $n\in \mathbb{N}$, we have
\begin{equation}\label{eq: 3.1}
A^{n}=\sum_{j=0}^{m_{1}-1}\delta_{nj}B_{1j}+\sum_{j=0}^{m_{2}-1}\binom{n}{j}\alpha_{2}^{n-j}B_{2j}+\cdots+\sum_{j=0}^{m_{s}-1}\binom{n}{j}\alpha_{s}^{n-j}B_{sj}
\end{equation}
where $B_{jk_{j}}=L_{jk_{j}}(A)[\chi_{A}]$ and $\delta_{nj}$ denotes the Kronecker symbol.
\end{theorem}
\begin{proof}
Using Theorem~\ref{Thm 2.7}, we can easily prove this result.
\end{proof}
\begin{remark} It is easy to verify the following
\begin{equation*}
B_{jk_{j}}B_{jm_{j}-k_{j}}=0
\end{equation*}
\end{remark}
Using the previous theorem we find a new method to calculate the Chevalley--Jordan decomposition and the spectral projections of $A$ at the same time.
\begin{theorem}\label{Thm 3.3}
Let $A$ be a $k\times k$ matrix, and let $\chi_{A}(x)=(x-\alpha_{1})^{m_{1}}(x-\alpha_{2})^{m_{2}}\cdots(x-\alpha_{s})^{m_{s}}$, $\alpha_{1}=0$, be its characteristic polynomial (possibly $m_{1}=0$). Then the Chevalley--Jordan decomposition of $A$ is
\begin{equation}\label{eq: 3.5}
A=\pmb{D}+\pmb{N}
\end{equation}
where
\begin{equation}\label{eq: 3.6}
\pmb{N}=B_{11}+B_{21}+\cdots+B_{s1}
\end{equation}
and
\begin{equation}\label{eq: 3.7}
\pmb{D}=\alpha_{2}B_{20}+\alpha_{3}B_{30}+\cdots+\alpha_{s}B_{s0}
\end{equation}
where $B_{jr}=L_{jr}(A)[\chi_{A}]$ for $j=1,2,\ldots,s$ and $r=0,1$.
Moreover, $B_{10},B_{20},\ldots,B_{s0}$ are the spectral projections of $A$ at $\alpha_{1},\alpha_{2},\ldots,\alpha_{s}$, respectively.
\end{theorem}
\begin{proof}
To obtain
\begin{equation*}
A=\pmb{D}+\pmb{N}
\end{equation*}
and
\begin{equation}\label{eq: 3.8}
I=B_{10}+B_{20}+\cdots+B_{s0},
\end{equation}
it suffices to take $n=1$ and $n=0$ in Formula~\eqref{eq: 3.1}, respectively. Since $B_{j0},B_{j1},j=1,2,\ldots,s$, are polynomials of the matrix $A$, we have $\pmb{D}\pmb{N}=\pmb{N}\pmb{D}$.\\
On the other hand, it is clear that $B_{j0}B_{i0}=B_{i0}B_{j0}=0, i\neq j$, is a particular case of Formula~\eqref{eq: 2.21}. Multiplying both sides
of~\eqref{eq: 3.8} by $B_{j0}$, we get $B_{j0}^{2}=B_{j0},j=1,2,\ldots,s$.\\
Furthermore, the matrix $B_{j0}$ is diagonalizable, then so is $\alpha_{j}B_{j0}$. The fact that $\alpha_{j}B_{j0}$ and $\alpha_{i}B_{i0}$ commute assures then that $\pmb{D}$ is diagonalizable.\\
To complete the proof, it remains to show that the matrix $\pmb{N}$ is nilpotent. To see this, it suffices to verify, using Formula~\eqref{eq: 1.1}, that each $B_{j1}$ is nilpotent.
\end{proof}
\begin{remarks}~
\begin{enumerate}[1.]
\item We can immediately find that $B_{j0},j=1,2,\ldots,s$, are the spectral projections of $A$ by using a result of~\cite{Mmou}.
\item Formulas~\eqref{eq: 3.5},\eqref{eq: 3.6}, and \eqref{eq: 3.7} can be also shown using the matrix equation~\eqref{eq: xyz}.
\end{enumerate}
\end{remarks}
The following example is an illustration of Theorem~\ref{Thm 3.3}.
\begin{example}
Let us consider the matrix used in Example~\ref{ex 2.12}
\begin{equation*}
A=\begin{pmatrix}
2 &  0 & 1   \\
   0 & 2 &0 \\
   0 & 0 & 3
\end{pmatrix}
\end{equation*}
The Chevalley--Jordan decomposition of this matrix is
\begin{equation*}
A=\pmb{D}+\pmb{N}
\end{equation*}
where
\begin{equation*}
\pmb{D}=2B_{10}+3B_{20}\quad\text{and}\quad \pmb{N}=B_{11}
\end{equation*}
A trivial calculation yields
\begin{align*}
B_{10}&=L_{10}(A)[\chi_{A}]=\begin{pmatrix}
1 &  0 & -1   \\
   0 & 1 &0 \\
   0 & 0 & 0
\end{pmatrix}\\
B_{11}&=L_{11}(A)[\chi_{A}]=\begin{pmatrix}
0 &  0 & 0   \\
   0 & 0 &0 \\
   0 & 0 & 0
\end{pmatrix}\\
B_{20}&=L_{20}(A)[\chi_{A}]=\begin{pmatrix}
0 &  0 & 1   \\
   0 & 0 &0 \\
   0 & 0 & 1
\end{pmatrix}\\
\end{align*}
\end{example}
The following Theorem shows that knowing only the associated matrices $B_{ij}$ of $A$, we can simply provide the minimal polynomial of the matrix $A$ and the positive powers of its Drazin inverse.
\begin{theorem}\label{thm: aaa}
Let $A$ be a matrix and let $\chi_{A}(x)=(x-\alpha_{1})^{m_{1}}(x-\alpha_{2})^{m_{2}}\cdots(x-\alpha_{s})^{m_{s}}$ be its characteristic polynomial with $\alpha_{1}=0$ (possibly $m_{1}=0$). Then
\begin{enumerate}[1.]
\item The index of $\alpha_{i}$ is the greatest integer $j$ such that $B_{ij-1}\neq0$.
\item The index of $\alpha_{i}$ is $1$ if and only if $B_{i1}=0$.
\item $A$ is diagonalizable if and only if $B_{i1}=0, i=1,2,\ldots,s$.
\item For all positive integer $n$, we have
\begin{equation*}
A_{D}^{n}=\sum_{j=0}^{m_{2}-1}\binom{-n}{j}\alpha_{2}^{-n-j}B_{2j}+\cdots+\sum_{j=0}^{m_{s}-1}\binom{-n}{j}\alpha_{s}^{-n-j}B_{sj}
\end{equation*}
where $A_{D}$ denotes the Drazin inverse of $A$.
\end{enumerate}
\end{theorem}
\begin{proof}
Clearly formula~\eqref{eq: 3.1} is a $\mathcal{P}$-canonical form of $A$ (see~\cite{Mmou}) and the result follows then from Corollary~$3.6.$ and Theorem~$3.9.$ of~\cite{Mmou} and Theorem~\ref{Thm 3.1} above.
\end{proof}
\begin{remark}
Let $r_{j}$ be the index of $\alpha_{j}$. In view of Property $1.$ of Theorem~\ref{thm: aaa}, Formula~\eqref{eq: 2.20} becomes
\begin{equation*}
\e^{tA}=\sum_{j=1}^{s}\sum_{k_{j}=0}^{r_{j}-1}\frac{t^{k_{j}}}{k_{j}!}\e^{\alpha_{j}t}B_{jk_{j}}
\end{equation*}
\end{remark}
To constitute an illustration of Theorem~\ref{thm: aaa}, let us consider the same example given in~\cite{Lzhang}
\begin{example}
Let us determine the Drazin inverse and its powers of the following matrix
$$A=\begin{pmatrix}
2&0&0\\-1&1&1\\-1&-1&-1
\end{pmatrix}$$
The characteristic polynomial of $A$ is $\chi_{A}(x)=x^{2}(x-2)$.\\
Using the formula of Theorem \ref{thm: aaa}, we obtain for all $n\geq 1$
\begin{equation*}
A_{D}^{n}=\binom{-n}{0}2^{-n}B_{20}
\end{equation*}
On the other hand, we have
\begin{equation*}
B_{20}=\frac{1}{4}A^{2}
\end{equation*}
Simple calculation gives
$$
B_{20}=\begin{pmatrix}
1&0&0\\-1&0&0\\0&0&0
\end{pmatrix}$$
Therefore
\begin{equation*}
A_{D}^{n}=\begin{pmatrix}
2^{-n}&0&0\\-2^{-n}&0&0\\0&0&0
\end{pmatrix}, \,\,\ n\geq 1
\end{equation*}
for $n=1$ the Drazin inverse of the matrix $A$ is
\begin{equation*}
\renewcommand{\arraystretch}{1.4}
A_{D}=\begin{pmatrix}
\frac{1}{2}&0&0\\ -\frac{1}{2}&0&0\\0&0&0
\end{pmatrix}
\end{equation*}
which is in agreement with the result in~\cite{Lzhang}.
\end{example}
The following example appears in~\cite{Eehartwig} it is used here to illustrate our result on the Drazin inverse
\begin{example}
Let $$A=\begin{pmatrix}
1&-1&1&1\\0&1&-1&1\\1&-1&1&2 \\1&-1&1&1
\end{pmatrix}$$
be a matrix with characteristic polynomial $\chi_{A}(x)=x^{2}(x-(2+\sqrt{2}))(x-(2-\sqrt{2}))$.\\
Using the formula of Theorem~\ref{thm: aaa}, we obtain for all $n\geq 1$
\begin{equation*}
A_{D}^{n}=\binom{-n}{0}(2+\sqrt{2})^{-n}B_{20}+\binom{-n}{0}(2-\sqrt{2})^{-n}B_{30}
\end{equation*}
On the other hand, we have
$$
\left \{
\begin{array}{rcl}
B_{20}&=&\frac{3\sqrt{2}-4}{8}A^{2}(A-(2-\sqrt{2})I) \vspace*{0.5pc}\\
B_{30}&=&\frac{3\sqrt{2}+4}{-8}A^{2}(A-(2+\sqrt{2})I)
\end{array}
\right.
$$
Simple calculation gives
\renewcommand{\arraystretch}{1.5}
$$
B_{20}=\begin{pmatrix}
\frac{1}{4}&\frac{-\sqrt{2}}{4}&\frac{\sqrt{2}}{4}&\frac{1}{4}\\\frac{4-3\sqrt{2}}{8}&\frac{3-2\sqrt{2}}{4}&\frac{2\sqrt{2}-3}{4}&\frac{4-3\sqrt{2}}{8}\\ \frac{4-\sqrt{2}}{8}&\frac{1-2\sqrt{2}}{4}&\frac{2\sqrt{2}-1}{4}&\frac{4-\sqrt{2}}{8}\\ \frac{1}{4}&\frac{-\sqrt{2}}{4}&\frac{\sqrt{2}}{4}&\frac{1}{4}
\end{pmatrix}$$
and
$$
B_{30}=\begin{pmatrix}
\frac{1}{4}&\frac{\sqrt{2}}{4}&\frac{-\sqrt{2}}{4}&\frac{1}{4}\\\frac{3\sqrt{2}+4}{8}&\frac{3+2\sqrt{2}}{4}&\frac{-3-2\sqrt{2}}{4}&\frac{3\sqrt{2}+4}{8}\\ \frac{4+\sqrt{2}}{8}&\frac{1+2\sqrt{2}}{4}&\frac{-2\sqrt{2}-1}{4}&\frac{4+\sqrt{2}}{8}\\ \frac{1}{4}&\frac{\sqrt{2}}{4}&\frac{-\sqrt{2}}{4}&\frac{1}{4}
\end{pmatrix}$$
It follows that
\begin{equation*}
A_{D}^{n}=(2+\sqrt{2})^{-n}B_{20}+(2-\sqrt{2})^{-n}B_{30}, \,\,\ n\geq 1
\end{equation*}
for $n=1$ the Drazin inverse of the matrix $A$ is
\renewcommand{\arraystretch}{1.2}
\begin{equation*}
A_{D}=\frac{1}{4}\begin{pmatrix}
2&2&-2&2\\ 7&10&-10&7\\ 5&6&-6&5\\ 2&2&-2&2
\end{pmatrix}
\end{equation*}
\end{example}
\section{Explicit formulas for logarithms of matrices}
\label{sec:4}
In this section, we derive some explicit and elegant formulas for logarithms of matrices with the aid of the results of Section~\ref{sec:3} and Theorem $4.2$ of~\cite{Mouc22}.
\begin{theorem}\label{Thhh}
Let $A$ be a $k\times k$ nonsingular matrix, and let $\chi_{A}(x)=(x-\alpha_{1})^{m_{1}}(x-\alpha_{2})^{m_{2}}\cdots(x-\alpha_{s})^{m_{s}}$ be its characteristic polynomial where $\alpha_{1}=\e^{\beta_{1}},\ldots,\alpha_{s}=\e^{\beta_{s}}$. Then the matrix
\begin{equation}\label{eq: ffft}
\sum_{p=1}^{s}\beta_{p}B_{p0}+\sum_{p=1}^{s}\sum_{j=1}^{m_{p}-1}\Big(\frac{(-1)^{j-1}}{j\alpha_{p}^{j}}\Big)B_{pj},
\end{equation}
where $B_{jk_{j}}=L_{jk_{j}}(A)[\chi_{A}]$, is a logarithm of $A$.
\end{theorem}
\begin{proof}
By Theorem~\ref{Thm 3.1}, we have for all $n\in \mathbb{N}$
\begin{equation*}
A^{n}=\sum_{j=0}^{m_{1}-1}\binom{n}{j}\e^{\beta_{1}(n-j)}B_{1j}+\cdots+\sum_{j=0}^{m_{s}-1}\binom{n}{j}\e^{\beta_{s}(n-j)}B_{sj}
\end{equation*}
If we plug $t$ into this formula for $n$, we obtain the following smooth matrix function
\begin{equation*}
A(t)=\sum_{j=0}^{m_{1}-1}\binom{t}{j}\e^{\beta_{1}(t-j)}B_{1j}+\cdots+\sum_{j=0}^{m_{s}-1}\binom{t}{j}\e^{\beta_{s}(t-j)}B_{sj}.
\end{equation*}
Since
\begin{equation*}
\binom{t}{j}=
\frac{t(t-1)\cdots(t-j+1)}{j!},
\end{equation*}
for all positive integer $j$,
the derivative at $0$ of the function $\binom{t}{j}$ is
\begin{equation*}
\binom{t}{j}^\prime(0)=\dfrac{(-1)^{j-1}}{j}.
\end{equation*}
Consequently, using the convention that $\binom{t}{0}=1$, we obtain
\begin{equation*}
A'(0)=\sum_{p=1}^{s}\beta_{p}B_{p0}+\sum_{p=1}^{s}\sum_{j=1}^{m_{p}-1}\dfrac{(-1)^{j-1}}{j\alpha_{p}^{j}}B_{pj}.
\end{equation*}
Then the conclusion follows from Theorem~$4.2$ of~\cite{Mouc22}.
\end{proof}
As an illustration of Theorem~\ref{Thhh}, consider the following example.
\begin{example}
Let $$A=\begin{pmatrix}
3&0&0\\2&3&0\\1&4&2
\end{pmatrix}$$
The characteristic polynomial of $A$ is $\chi_{A}(x)=(x-3)^{2}(x-2)$.\\
Applying Formula~\eqref{eq: ffft}, we obtain that
\begin{equation*}
C=\ln(3)B_{10}+\ln(2)B_{20}+\frac{1}{3}B_{11}
\end{equation*}
is the principal logarithm of $A$, where
$$
\left \{
\begin{array}{rcl}
B_{1j}&=&(A-3I)^{j}(A-2I)\displaystyle\sum_{i=0}^{1-j}(-1)^{i}(A-3I)^{i}, j=0,1\\
B_{20}&=&(A-3I)^{2}
\end{array}
\right.
$$
Simple calculation gives
$$
\left \{
\begin{array}{rcl}
B_{10}&=&-A^{2}+6A-8I\\
B_{11}&=&A^{2}-5A+6I\\
B_{20}&=&A^{2}-6A+9I\\
\end{array}
\right.
$$
Finally, we obtain
$$C=\begin{pmatrix}
\ln(3)&0&0\vspace*{0.3pc}\\\frac{2}{3}&\ln(3)&0\vspace*{0.3pc}\\7\ln(\frac{2}{3})+\frac{8}{3}&4\ln(\frac{3}{2})&\ln(2)
\end{pmatrix}$$
\end{example}
\begin{corollary}
If $A$ is a nonsingular matrix and  $\chi_{A}(x)=(x-\alpha_{1})(x-\alpha_{2})\cdots(x-\alpha_{k})$ its characteristic polynomial with distinct roots $\alpha_{1}=\e^{\beta_{1}},\ldots, \alpha_{k}=\e^{\beta_{k}}$, then the matrix
\begin{equation*}
\sum_{j=1}^{k}\beta_{j}B_{j},
\end{equation*}
where $B_{j}=\prod\limits_{i=1,i\neq j}^{k}\dfrac{1}{\alpha_{j}-\alpha_{i}}(A-\alpha_{i}I)$, is a logarithm of $A$.
\end{corollary}
\begin{proof}
Is a direct consequence of Theorem~\ref{Thhh} and the fact that in the present case, $$B_{j0}=\prod_{i=1,i\neq j}^{k}\dfrac{1}{\alpha_{j}-\alpha_{i}}(A-\alpha_{i}I).$$
\end{proof}
\begin{example}
To illustrate the previous corollary, we consider the same matrix that in~\cite{Jamarrero}
\renewcommand{\arraystretch}{1.4}
$$A=\begin{pmatrix}
0&\frac{5}{16}&\frac{-1}{32}\\ -1&1&0\\0&0&1
\end{pmatrix}$$
The characteristic polynomial of $A$ is $\chi_{A}(x)=(x-1)(x-(\frac{1}{2}+\frac{1}{4}i))(x-(\frac{1}{2}-\frac{1}{4}i))$.\\
Let $\alpha_{1}=1,\alpha_{2}=\frac{1}{2}+\frac{1}{4}i$ and $\alpha_{3}=\frac{1}{2}-\frac{1}{4}i$. It is clear that the eigenvalues of $A$ are all not in $\mathbb{R}^{-}$, then the principal logarithm of $A$ exists. To obtain this logarithm we must take the principal logarithm of $\alpha_{1},\alpha_{2}$ and $\alpha_{3}$ which are $0,ln(\frac{\sqrt{5}}{4})+i\arctan(\frac{1}{2})$ and $ln(\frac{\sqrt{5}}{4})-i\arctan(\frac{1}{2})$ respectively. Applying the previous result, we obtain that
\begin{equation}\label{ex rachidi}
B=\beta_{2}B_{2}+\beta_{3}B_{3}
\end{equation}
is the principal logarithm of $A$, where
$$
\left \{
\begin{array}{rcl}
B_{2}&=&\dfrac{1}{(\alpha_{2}-\alpha_{1})(\alpha_{2}-\alpha_{3})}(A-\alpha_{1}I)(A-\alpha_{3}I)\vspace*{0.5pc}\\
B_{3}&=&\dfrac{1}{(\alpha_{3}-\alpha_{1})(\alpha_{3}-\alpha_{2})}(A-\alpha_{1}I)(A-\alpha_{2}I)\vspace*{0.5pc}\\
\beta_{2}&=&ln(\frac{\sqrt{5}}{4})+i\arctan(\frac{1}{2})\vspace*{0.5pc}\\
\beta_{3}&=&ln(\frac{\sqrt{5}}{4})-i\arctan(\frac{1}{2})
\end{array}
\right.
$$
Simple calculation gives
$$
\left \{
\begin{array}{rcl}
B_{2}&=&\begin{pmatrix}
\frac{1}{2}+i&\frac{-5}{8}i&\frac{i}{16}\\ 2i&\frac{1}{2}-i&\frac{-1}{20}+\frac{i}{10}\\0&0&0
\end{pmatrix}\vspace*{0.5pc}\\
B_{3}&=&\begin{pmatrix}
\frac{1}{2}-i&\frac{5}{8}i&-\frac{i}{16}\\ -2i&\frac{1}{2}+i&\frac{-1}{20}+\frac{-i}{10}\\0&0&0
\end{pmatrix}
\end{array}
\right.
$$
Substituting $B_{2},B_{3},\beta_{2}$ and $\beta_{3}$ by their values in~\eqref{ex rachidi}, we obtain
$$B=ln\big(\frac{\sqrt{5}}{4}\big)\begin{pmatrix}
1&0&0\\0&1&\frac{-1}{10}\\0&0&0
\end{pmatrix}+\arctan\big(\frac{1}{2}\big)\begin{pmatrix}
-2&\frac{5}{4}&\frac{-1}{8}\\ -4&2&\frac{-1}{5}\\0&0&0
\end{pmatrix}$$
\end{example}
\begin{remark}
Let $A$ be a nonsingular matrix of order $2$ and $\alpha_{1},\alpha_{2}$ its eigenvalues.
\begin{enumerate}
\item[i)] If $\alpha_{1}\neq \alpha_{2}$ then $log(A)=\dfrac{log(\alpha_{1})-log(\alpha_{2})}{\alpha_{1}-\alpha_{2}}A+\dfrac{\alpha_{1}log(\alpha_{2})-
    \alpha_{2}log(\alpha_{1})}{\alpha_{1}-\alpha_{2}}I$.\vspace*{0.5pc}\\
\item[ii)] If $\alpha_{1}=\alpha_{2}=\alpha$ then $log(A)=\dfrac{1}{\alpha}A+(log(\alpha)-1)I$.
\end{enumerate}
\end{remark}
\begin{corollary}
If $A$ is a nonsingular matrix with characteristic polynomial $\chi_{A}(x)=(x-\alpha)^{k}$, $\alpha=\e^{\beta}$, then the matrix
\begin{equation*}
\beta B_{0}+\sum_{j=1}^{k-1}\frac{(-1)^{j-1}}{j\alpha^{j}}B_{j},
\end{equation*}
where $B_{j}=(A-\alpha I)^{j}$, is a logarithm of $A$.
\end{corollary}
\begin{proof}
Follows immediately from Theorem~\ref{Thhh}.
\end{proof}
\begin{example}
For any complex number $a$, let us calculate the logarithm of the following matrix
\begin{equation*}
A(a)= \begin{pmatrix}
    1 &  0 &\cdots & 0   \\
    a & 1 & \cdots &  0 \\
     a^{2} &2a & \cdots & 0 \\
      a^{3} &3a^{2} & \cdots & 0 \\
     \binom{4}{0}a^{4} &\binom{4}{1}a^{3} & \cdots & 0 \\
    \vdots & \vdots & \cdots & \vdots \\
    \binom{k-1}{0}a^{k-1} & \binom{k-1}{1}a^{k-2}  & \cdots & 1
\end{pmatrix}
\end{equation*}
The characteristic polynomial of $A(a)$ is $\chi_{A(a)}(x)=(x-1)^{k}$.\\
In this example we have $\alpha=\e^{0}$. Applying the previous result, we obtain that
\begin{equation*}
B(a)=\sum_{j=1}^{k-1}\frac{(-1)^{j-1}}{j}(A(a)-I)^{j}
\end{equation*}
is the principal logarithm of $A(a)$. We have
\begin{equation*}
(A(a)-I)^{j}=\sum_{i=0}^{j}\binom{j}{i}(-1)^{j-i}A(a)^{i}
\end{equation*}
It is easily seen that
\begin{equation*}
A(a)^{i}=A(ia)
\end{equation*}
Thus
\begin{equation*}
(A(a)-I)^{j}=\sum_{i=0}^{j}\binom{j}{i}(-1)^{j-i}A(ia)
\end{equation*}
The $(l,q)$-th entry of the matrix $A(ia)$ is
\[  (A(ia))_{lq}= \left\{ \begin{array}{ll}
         \binom{l-1}{q-1}(ia)^{l-q} & \mbox{if $l \geq q$}\\
        0, & \mbox{otherwise },\end{array} \right. \]
Therefore
\[  (A(a)-I)^{j})_{lq}= \left\{ \begin{array}{ll}
         \binom{l-1}{q-1}a^{l-q}\sum_{i=0}^{j}\binom{j}{i}(-1)^{j-i}i^{l-q} & \mbox{if $l \geq q$}\\
        0, & \mbox{otherwise },\end{array} \right. \]
It is easy to show that if $l \geq q$, then
\[  \sum_{i=0}^{j}\binom{j}{i}(-1)^{j-i}i^{l-q}= \left\{ \begin{array}{ll}
        qa  & \mbox{if $l=q+1$}\\
        0, & \mbox{otherwise },\end{array} \right. \]
so the principal logarithm of $A(a)$ is
$$B(a)=\begin{pmatrix}
  0  & 0  & 0 & 0 & \cdots & 0 \\
  a & 0 & 0 & 0 & \cdots & 0 \\
  0 & 2a & 0 & 0 & \cdots & 0\\
  0 & 0 & 3a & 0 & \cdots & 0 \\
  \vdots & \vdots & \ddots & \ddots & \ddots & \vdots \\
  0 & 0 & \cdots & 0& (k-1)a & 0 \\
 \end{pmatrix}$$
\end{example}
The following interesting corollaries are consequence of Theorem~\ref{Thhh}.
\begin{corollary}
If $\chi_{A}(x)=(x-\alpha_{1})^{m_{1}}(x-\alpha_{2})^{m_{2}}$, $\alpha_{1}=\e^{\beta_{1}}$ and $\alpha_{2}=\e^{\beta_{2}}$ are nonzero two distinct complex numbers, then the matrix
\begin{equation*}
\beta_{1}B_{10}+\beta_{2}B_{20}+\sum_{j=1}^{m_{1}-1}\frac{(-1)^{j-1}}{j\alpha_{1}^{j}}B_{1j}+\sum_{j=1}^{m_{2}-1}\frac{(-1)^{j-1}}{j\alpha_{2}^{j}}B_{2j}
\end{equation*}
is a logarithm of $A$, where
$$ \left \{
\begin{array}{rcl}
B_{1j}&=&(A-\alpha_{1}I)^{j}(A-\alpha_{2}I)^{m_{2}}\displaystyle \sum_{i=0}^{m_{1}-j-1}\frac{(-1)^{i}\binom{m_{2}+i-1}{m_{2}-1}}{(\alpha_{1}-\alpha_{2})^{m_{2}+i}}(A-\alpha_{1}I)^{i}\\
B_{2j}&=&(A-\alpha_{1}I)^{m_{1}}(A-\alpha_{2}I)^{j}\displaystyle \sum_{i=0}^{m_{2}-j-1}\frac{(-1)^{i}\binom{m_{1}+i-1}{m_{1}-1}}{(\alpha_{2}-\alpha_{1})^{m_{1}+i}}(A-\alpha_{2}I)^{i}
\end{array}
\right.
$$
\end{corollary}
\begin{corollary}
If $\chi_{A}(x)=(x-\alpha_{1})^{m_{1}}(x-\alpha_{2})$, $\alpha_{1}=\e^{\beta_{1}}$ and $\alpha_{2}=\e^{\beta_{2}}$ are nonzero two distinct
complex numbers, then the matrix
\begin{equation*}
\beta_{1}B_{10}+\beta_{2}B_{20}+\sum_{j=1}^{m_{1}-1}\frac{(-1)^{j-1}}{j\alpha_{1}^{j}}B_{1j}
\end{equation*}
is a logarithm of $A$, where
$$
\left \{
\begin{array}{rcl}
B_{1j}&=&(A-\alpha_{1}I)^{j}(A-\alpha_{2}I)\displaystyle\sum_{i=0}^{m_{1}-j-1}\frac{(-1)^{i}}{(\alpha_{1}-\alpha_{2})^{i+1}}(A-\alpha_{1}I)^{i}\\
B_{20}&=&\displaystyle\frac{1}{(\alpha_{2}-\alpha_{1})^{m_{1}}}(A-\alpha_{1}I)^{m_{1}}
\end{array}
\right.
$$
\end{corollary}
More generally, we have the following result.
\begin{corollary}
If $\chi_{A}(x)=(x-\alpha_{1})^{m_{1}}\prod\limits_{j=2}^{s}(x-\alpha_{j})$, $\alpha_{1}=\e^{\beta_{1}},\ldots, \alpha_{s}=\e^{\beta_{s}}$ are nonzero distinct complex numbers, then the matrix
\begin{equation}\label{eq: 4.7}
\beta_{1}B_{10}+\sum_{j=1}^{m_{1}-1}\frac{(-1)^{j-1}}{j\alpha_{1}^{j}}B_{1j}+\sum_{j=2}^{s}\frac{\beta_{j}}{P_{j}(\alpha_{j})}P_{j}(A)
\end{equation}
is a logarithm of $A$, where
$P=\chi_{A},$
$$B_{1j}=(A-\alpha_{1}I)^{j}\prod\limits_{l=2}^{s}(A-\alpha_{l}I)\sum_{i=0}^{m_{1}-j-1}\sum_{l=2}^{s}\frac{(-1)^{i}
a_{l}}{(\alpha_{1}-\alpha_{l})^{i+1}}(A-\alpha_{1}I)^{i}$$
and
\begin{align*}
a_{l}=\displaystyle\begin{cases}
\frac{1}{\prod\limits_{p=2,p\neq l}^{s}(\alpha_{l}-\alpha_{p})} &\text{if}\quad s\geq3\\
1 &\text{if}\quad s=2
\end{cases}
\end{align*}
\end{corollary}
\begin{example}
We now illustrate this case by computing the principal logarithm of the following matrix
$$A=\begin{pmatrix}
1&2&4&1&0\\0&1&2&0&0\\0&0&2&1&0\\0&0&0&2&1\\0&0&0&2&3
\end{pmatrix}$$
The characteristic polynomial of $A$ is $\chi_{A}(x)=(x-1)^{3}(x-2)(x-4)$.\\
Using the result of the previous corollary, we obtain that the matrix
\begin{equation*}
C=B_{11}-\frac{1}{2}B_{12}+\frac{ln(2)}{P_{2}(2)}P_{2}(A)+\frac{ln(4)}{P_{3}(4)}P_{3}(A)
\end{equation*}
is the principal logarithm of $A$. On the other hand, we can easily check that
$$
\left \{
\begin{array}{rcl}
B_{11}&=&(A-I)(A-2I)(A-4I)(\frac{4}{9}A-\frac{1}{9}I)\vspace*{0.5pc}\\
B_{12}&=&\frac{1}{3}(A-I)^{2}(A-2I)(A-4I)\vspace*{0.5pc}\\
P_{2}&=&(A-I)^{3}(A-4I)\vspace*{0.5pc}\\
P_{3}&=&(A-I)^{3}(A-2I)\\
\end{array}
\right.
$$
Thus the principal logarithm of $A$ is
\begin{equation*}
\renewcommand{\arraystretch}{2.2}
C=\begin{pmatrix}
0&2&8ln(2)-4&\frac{130}{27}ln(2)-\frac{26}{9}&\frac{19}{9}-\frac{86}{27}ln(2)\\
0&0&2ln(2)&\frac{11}{9}ln(2)-\frac{4}{3}&\frac{2}{3}-\frac{7}{9}ln(2)\\
0&0&ln(2)&\frac{5}{6}ln(2)&-\frac{1}{6}ln(2)\\
0&0&0&\frac{2}{3}ln(2)&\frac{2}{3}ln(2)\\
0&0&0&\frac{4}{3}ln(2)&\frac{4}{3}ln(2)
\end{pmatrix}
\end{equation*}
\end{example}
\section{Conclusion}
We have presented a new and elegant method to facilitate the computation of the matrix exponential function. In addition, a closed-form formula for the $n$th power of an arbitrary complex matrix $A$ is provided. This formula allows us to deduce the Chevalley--Jordan decomposition and the spectral projections of $A$. We also deduce from this formula explicit and elegant formulas for the computation of the logarithm and the Drazin inverse of matrices.

\end{document}